\documentclass[11pt]{amsart}

\usepackage[T1]{fontenc}
\usepackage{lmodern}
\usepackage{amsmath,amssymb,amsthm,mathtools}
\usepackage{microtype}
\usepackage[hidelinks]{hyperref}

\hypersetup{
  pdftitle={Prescribed extension spectra of mock automorphisms over finite fields},
  pdfauthor={Stefan Barańczuk, Tomasz Ślusarski},
  pdfsubject={Extension spectra, mock automorphisms, and Keller maps over finite fields},
  pdfkeywords={finite fields, Keller maps, mock automorphisms, exceptional polynomials, permutation polynomials}
}

\newtheorem{theorem}{Theorem}[section]
\newtheorem{proposition}[theorem]{Proposition}
\newtheorem{lemma}[theorem]{Lemma}
\newtheorem{corollary}[theorem]{Corollary}
\theoremstyle{definition}
\newtheorem{definition}[theorem]{Definition}
\newtheorem{question}[theorem]{Question}
\theoremstyle{remark}
\newtheorem{remark}[theorem]{Remark}
\newtheorem{example}[theorem]{Example}

\newcommand{\F}{\mathbf{F}}
\newcommand{\A}{\mathbb{A}}
\newcommand{\N}{\mathbf{N}}
\newcommand{\Jac}{\operatorname{Jac}}
\newcommand{\Tr}{\operatorname{Tr}}
\newcommand{\Norm}{\operatorname{N}}
\newcommand{\id}{\operatorname{id}}
\newcommand{\Specper}{\mathcal{S}}
\newcommand{\lcm}{\operatorname{lcm}}

\title[Prescribed extension spectra of mock automorphisms]
{Prescribed extension spectra of mock automorphisms over finite fields}

\author{Stefan Barańczuk, Tomasz Ślusarski}
\address{Adam Mickiewicz University, Poznań, Poland}
\email{stefan.baranczuk@amu.edu.pl} \email{tomasz.slusarski@amu.edu.pl}


\begin{document}

\begin{abstract}
We classify the finite-extension permutation spectra occurring in an explicit family of mock automorphisms over finite fields. Every finite nonempty divisor-closed subset of $\N_{>0}$ occurs; in particular, this disproves Maubach and Willems’ conjecture that every mock automorphism permutes infinitely many finite extensions.

Every finite-spectrum map satisfying
\[
 \Jac(F)=I_n,\qquad F|_{\F_q^n}=\id
\]
has total degree at least $p(q+1)$. For every $n\geq2$, this bound is
attained by a map with spectrum $\{1\}$. In dimension one, the same bound
and exact spectrum $\{1\}$ are attained over every odd field and over
$\F_2$; for even $q>2$ we give a nonexceptional construction with an
effective finite bound for its spectrum. If $\delta_p$ denotes the least degree of a nonlinear reduced permutation polynomial over $\mathbb F_p$, then the least degree $\mu_p$ of a nonexceptional one-variable mock automorphism satisfies
\[
\mu_2=6,\qquad \mu_3=12,\qquad \mu_p=p\delta_p\quad(p\ge5).
\] We also record a cofinite-spectrum criterion and isolate
the surviving prime-to-$p$ question. Every nontrivial counterexample
constructed here has geometric generic degree divisible by $p$.
\end{abstract}

\maketitle

\section{Introduction}

Let $k$ be a field and let
\[
 F=(F_1,\ldots,F_n)\in k[X_1,\ldots,X_n]^n
\]
be a polynomial endomorphism of affine $n$-space. It is a \emph{Keller
map} if
\[
 \det\Jac(F)\in k^\times.
\]

Maubach and Willems introduced the following finite-field notion
\cite[Section~2]{MaubachWillems2014}.

\begin{definition}\label{def:mock}
A polynomial map $F\in\F_q[X_1,\ldots,X_n]^n$ is a \emph{mock
automorphism over $\F_q$} if $F$ is a Keller map and the induced function
\[
 F\colon\F_q^n\longrightarrow\F_q^n
\]
is bijective.
\end{definition}

A polynomial automorphism remains bijective after every finite scalar
extension. Maubach and Willems asked whether a mock automorphism retains
this property for infinitely many extensions.

\begin{quote}
\textbf{Maubach--Willems Conjecture 2.7.}
If $F$ is a mock automorphism over $\F_q$, then there are infinitely many
finite extensions $K/\F_q$ for which $F\colon K^n\to K^n$ is bijective.
\end{quote}

Immediately before the conjecture, they considered
\[
 (x,y)\longmapsto
 \bigl(x+y(x^p-x),\;y+y(x^p-x)\bigr).
\]
This map is the identity on $\F_p^2$ and is noninjective over every proper
finite extension, but it is not Keller. Their conjecture asks whether the
Jacobian condition rules out this behavior. They proved the conjecture for
a one-variable class of linearized polynomials
\cite[Example~2.6, Conjecture~2.7, and Lemma~2.8]{MaubachWillems2014}; see
also \cite{Berson2014} for the matrix theory of linearized polynomial maps.

The extension degrees on which a finite cover induces bijections are
classically organized into an \emph{exceptionality set}. In one variable this
is the usual exceptionality set of a polynomial; for finite separable covers it
is studied through the off-diagonal fiber product, arithmetic and geometric
monodromy, and extension of constants
\cite{Fried2005,GuralnickTuckerZieve2007}. Our spectrum terminology is not
intended as a new general notion. The multivariate maps below are generally
nonfinite, so the standard exceptional-cover formalism does not determine their
spectra directly.

The maps also lie in the positive-characteristic framework of
Borisov--Gabber--Vasiu. Writing $\operatorname{GA}_n(k)$ for the polynomial
automorphism group, they call an endomorphism $e$ \emph{basic} when the double
orbit
\[
 \operatorname{GA}_n(k)e\operatorname{GA}_n(k)
\]
contains an endomorphism represented by an $F_n$-system
\[
 X_i+h_i(X_1^p,\ldots,X_n^p),\qquad 1\leq i\leq n
\]
\cite[Definition~2.2(2)]{BorisovGabberVasiu2025}. The family studied here is
itself represented by such a system. Its étaleness and possible nonproperness
therefore belong to an established geometric setting; the new point is the
exact arithmetic control of its finite-extension bijectivity.

Our main contributions are as follows.

\begin{enumerate}
\item \textbf{Exact criterion and spectrum realization.} Let $e\geq1$, let
$n\geq2$, and let $g\in k[T]$ satisfy $g(0)=0$.
For
\[
 \Psi_{g,e,n}(X_1,\ldots,X_n)
 =\bigl(X_1+(X_2g(X_1))^{p^e},X_2,\ldots,X_n\bigr)
\]
we prove an exact injectivity criterion over every perfect overfield. Among the
mock automorphisms in this family, the possible spectra are exactly
$\N_{>0}$ and the finite nonempty divisor-closed subsets of $\N_{>0}$.
Thus every combinatorially admissible finite spectrum occurs, and spectrum
$\{1\}$ disproves Maubach--Willems Conjecture~2.7 over every finite field and in every
dimension $n\geq2$.

\item \textbf{Sharp degree threshold under strong normalization.}
For
\[
 \mathcal N_{q,n}=\{F:\Jac(F)=I_n,\ F|_{\F_q^n}=\id\},
\]
every map with finite spectrum has total degree at least $p(q+1)$. For every
$n\geq2$, the bound is attained by a spectrum-$\{1\}$ map. The class is
stable under affine conjugation over $\F_q$ but is substantially stronger than
the mock-automorphism condition.

\item \textbf{One-variable constructions.}
For every odd $q$, and separately for $q=2$, we construct a derivative-one
polynomial with spectrum exactly $\{1\}$. For even $q>2$, a
Carlitz--Lenstra--Wan obstruction gives a nonexceptional example together
with an effective finite search range for its spectrum.

\item \textbf{Least degrees and surviving repairs.}
Over prime fields we reduce the least-degree problem to the least degree
$\delta_p$ of a nonlinear reduced permutation polynomial. The uniform equality
$\mu_p=p\delta_p$ for $p\geq5$ is explicitly classification-dependent; broad
families are handled by elementary arguments and an explicit collision curve.
Finiteness does not repair the conjecture, whereas a cofinite spectrum
characterizes polynomial automorphisms among Keller maps. The prime-to-$p$
generic-degree question remains open and is related to work of Adjamagbo,
Maubach--Rauf, Lang, and Borisov--Gabber--Vasiu
\cite{Adjamagbo1995,MaubachRauf2017,Lang2025,BorisovGabberVasiu2025}.
\end{enumerate}

The exact criterion and realization theorem are elementary and independent of
exceptional-polynomial classification. The deep classification quoted in
Theorem~\ref{thm:fgsdegrees} enters only the uniform prime-field minimum
Theorem~\ref{thm:muexact}. The exact one-variable spectrum for even $q>2$,
the unrestricted multivariate degree minimum, and the prime-to-$p$ question
remain open.

\section{Extension spectra, collision schemes, and exceptional covers}
\label{sec:spectra}

Throughout, $q=p^r$ is a prime power.

\begin{definition}\label{def:spectrum}
For $F\in\F_q[X_1,\ldots,X_n]^n$, define its \emph{finite-extension
permutation spectrum} by
\[
 \Specper_q(F)
 =\{m\in\N_{>0}:F\colon\F_{q^m}^n\to\F_{q^m}^n
 \text{ is bijective}\}.
\]
A set $D\subseteq\N_{>0}$ is \emph{divisor-closed} if $m\in D$ and
$d\mid m$ imply $d\in D$. For arbitrary $n$, we call $F$
\emph{extension-exceptional} if $\Specper_q(F)$ is infinite. In one
variable this agrees with the standard term \emph{exceptional polynomial}.
For a multivariate polynomial map, $\deg F$ denotes the maximum of the total
degrees of its coordinate polynomials.
\end{definition}

The natural geometric object is the off-diagonal collision scheme
\[
 \operatorname{Col}(F)
 =\left(\A^n\times_{F}\A^n\right)\setminus\Delta.
\]
The diagonal is closed because polynomial morphisms are separated, so this
is an open subscheme of the fiber product.
For every field extension $K/\F_q$, its $K$-points are the ordered pairs of
distinct elements $x,y\in K^n$ satisfying $F(x)=F(y)$. Since
$\F_{q^m}^n$ is finite,
\begin{equation}\label{eq:collisioncriterion}
 m\in\Specper_q(F)
 \quad\Longleftrightarrow\quad
 \operatorname{Col}(F)(\F_{q^m})=\varnothing.
\end{equation}
For finite covers, this formulation leads to the classical monodromy theory
of exceptionality sets \cite{Fried2005,GuralnickTuckerZieve2007}. The maps
in Sections~\ref{sec:frob}--\ref{sec:degree} are generally nonfinite, which
is the principal reason that the finite-cover theory does not determine their
spectra.

\begin{proposition}[Spectrum invariance]\label{prop:invariance}
If $A,B\in\operatorname{Aut}_{\F_q}(\A^n)$, then
\[
 \Specper_q(A\circ F\circ B)=\Specper_q(F).
\]
\end{proposition}

\begin{proof}
After every finite scalar extension, $A$ and $B$ remain bijections. Hence
$A\circ F\circ B$ is bijective exactly when $F$ is bijective.
\end{proof}

\begin{proposition}\label{prop:divisorclosed}
For every polynomial map $F$ over $\F_q$, the set $\Specper_q(F)$ is
divisor-closed.
\end{proposition}

\begin{proof}
If $d\mid m$, then $\F_{q^d}\subseteq\F_{q^m}$. A map defined over
$\F_q$ preserves $\F_{q^d}^n$. If it is injective on $\F_{q^m}^n$,
then its restriction to the finite set $\F_{q^d}^n$ is injective and hence
bijective.
\end{proof}

For a one-variable polynomial $f\in\F_q[T]$, the set
$\Specper_q(f)$ is the classical exceptionality set. If $f$ is separable,
then $f$ is exceptional if and only if
\[
 \Phi_f(X,Y)=\frac{f(X)-f(Y)}{X-Y}
\]
has no absolutely irreducible factor defined over $\F_q$; see
\cite{FriedGuralnickSaxl1993,Fried2005,GuralnickTuckerZieve2007} and
\cite[Chapter~7]{LidlNiederreiter1997}. We use the following two standard
consequences.

\begin{theorem}[Carlitz--Lenstra--Wan]\label{thm:clw}
Let $f\in\F_q[T]$ be exceptional of degree $d>1$. Then
\[
 \gcd(d,q-1)=1.
\]
\end{theorem}

\begin{proof}[Reference]
This is Lenstra's theorem in the elementary form published by Cohen and
Fried \cite[Theorem~1.1]{CohenFried1995}.
\end{proof}

\begin{proposition}[Effective finiteness]\label{prop:effective}
Let $f\in\F_q[T]$ be separable, nonexceptional, and of degree $d>1$. If
$m\in\Specper_q(f)$, then
\[
 q^m<d^4.
\]
Consequently, $\Specper_q(f)$ can be determined by checking only
\[
 m<4\log_q d.
\]
\end{proposition}

\begin{proof}
The polynomial $\Phi_f$ has an absolutely irreducible factor over
$\F_q$, and absolute irreducibility remains true after every extension of
constants. Thus $f$ remains nonexceptional over $\F_{q^m}$. Von zur
Gathen's Theorem~1(i) states that a separable permutation polynomial of
degree $d>1$ over a field of order $Q$ is exceptional whenever
$Q\geq d^4$ \cite[Theorem~1(i)]{vonzurGathen1991}. Its contrapositive
gives $Q<d^4$ for a nonexceptional permutation polynomial. Apply this with
$Q=q^m$.
\end{proof}

\begin{remark}
Chahal and Ghorpade sharpen the numerical estimate through a precise Weil
bound for affine curves \cite{ChahalGhorpade2018}. The coarser $d^4$ bound
is sufficient here.
\end{remark}

\section{An exact Frobenius-perturbation criterion}
\label{sec:frob}

Let $k$ be a perfect field of characteristic $p>0$, let $e\geq1$, let
$n\geq2$, and let $g\in k[T]$ satisfy $g(0)=0$. Define
\begin{equation}\label{eq:psi}
 \Psi_{g,e,n}(X_1,\ldots,X_n)
 =\bigl(X_1+(X_2g(X_1))^{p^e},X_2,\ldots,X_n\bigr).
\end{equation}

\begin{proposition}[A representative of a basic endomorphism]\label{prop:basic}
The map $\Psi_{g,e,n}$ is a basic endomorphism in the sense of
Borisov--Gabber--Vasiu.
\end{proposition}

\begin{proof}
Write $g(T)=\sum_j a_jT^j$ and define
\[
 h(U,V)=V^{p^{e-1}}\sum_j a_j^{p^e}U^{jp^{e-1}}.
\]
Then
\[
 h(X_1^p,X_2^p)
 =X_2^{p^e}\sum_j a_j^{p^e}X_1^{jp^e}
 =\bigl(X_2g(X_1)\bigr)^{p^e}.
\]
Thus $\Psi_{g,e,n}$ is itself represented by the $F_n$-system
\[
 (X_1+h(X_1^p,X_2^p),X_2,\ldots,X_n).
\]
Its double orbit under left and right composition by polynomial automorphisms
therefore contains such a representative, so it is basic by
\cite[Definition~2.2(2)]{BorisovGabberVasiu2025}.
\end{proof}

\begin{remark}
Borisov, Gabber, and Vasiu study basic endomorphisms in a much broader
geometric setting. Proposition~\ref{prop:basic} locates the present family
in that theory. The new result below is not the \emph{étaleness} of the
family, but the exact description of its injectivity over each perfect
extension.
\end{remark}

\begin{theorem}[Exact perfect-field criterion]\label{thm:exactcriterion}
Let $L$ be a perfect field containing $k$. Then
\[
 \Jac(\Psi_{g,e,n})=I_n.
\]
Moreover, the following conditions are equivalent:
\begin{enumerate}
\item $\Psi_{g,e,n}\colon L^n\to L^n$ is injective;
\item $\Psi_{g,e,n}\colon L^n\to L^n$ is the identity;
\item $g(s)=0$ for every $s\in L$.
\end{enumerate}
If $L=\F_{q^m}$ and $g\in\F_q[T]$, these conditions are also equivalent to
\[
 T^{q^m}-T\mid g(T).
\]
\end{theorem}

\begin{proof}
All formal partial derivatives of a $p^e$th power vanish, so the Jacobian
matrix is $I_n$. If $g$ vanishes on $L$, then the map is visibly the identity.
Conversely, suppose that $g(s)\neq0$ for some $s\in L$. Since $L$ is perfect,
there is $u\in L$ with $u^{p^e}=-s$. Set
\[
 t=\frac{u}{g(s)}.
\]
Then
\[
 \Psi_{g,e,n}(s,t,0,\ldots,0)
 =\Psi_{g,e,n}(0,t,0,\ldots,0)
 =(0,t,0,\ldots,0).
\]
The two source points are distinct because $g(s)\neq g(0)=0$. Thus the map
is not injective.

For the final assertion, $T^{q^m}-T$ is squarefree and its roots are exactly
the elements of $\F_{q^m}$. Therefore it divides $g$ precisely when $g$
vanishes on that field.
\end{proof}

\begin{corollary}\label{cor:spectrumcriterion}
For $g\in\F_q[T]$ with $g(0)=0$,
\[
 \Specper_q(\Psi_{g,e,n})
 =\{m\geq1:T^{q^m}-T\mid g(T)\}.
\]
\end{corollary}

The criterion also has a stronger extension-theoretic consequence.

\begin{corollary}\label{cor:properperfect}
Assume $g\neq0$ and that the zero set of $g$ in a fixed algebraic closure
$\overline{k}$ is exactly $k$. Then $\Psi_{g,e,n}$ is the identity on
$k^n$ and is noninjective over every proper perfect overfield $L\supsetneq k$.
\end{corollary}

\begin{proof}
The map is the identity on $k^n$ because $g$ vanishes on $k$. Let $L$ be a
proper perfect overfield. We first show that every root of $g$ in $L$ lies
in $k$. Such a root $s$ is algebraic over $k$, since $g\neq0$. Its minimal
polynomial over $k$ divides $g$. A root of that minimal polynomial in
$\overline{k}$ is therefore a root of $g$, hence belongs to $k$ by
hypothesis. Irreducibility then forces the minimal polynomial to be linear,
so $s\in k$.

Choose $s\in L\setminus k$. Then $g(s)\neq0$, and
Theorem~\ref{thm:exactcriterion} gives a collision in $L^n$.
\end{proof}

\begin{remark}
The hypotheses of Corollary~\ref{cor:properperfect} force $k$ to be finite,
because a nonzero polynomial has only finitely many roots in a field. The
formulation over perfect overfields is retained because the resulting
noninjectivity statement also covers transcendental perfect extensions.
\end{remark}

\section{Realization of finite spectra}
\label{sec:realization}

Proposition~\ref{prop:divisorclosed} gives a necessary combinatorial
condition on every spectrum. For finite spectra, Theorem~\ref{thm:exactcriterion}
shows that it is sufficient.

\begin{theorem}[Main spectrum-realization theorem]\label{thm:realization}
Let $D\subseteq\N_{>0}$ be finite, nonempty, and divisor-closed. For every
$q$, every $e\geq1$, and every $n\geq2$, there is a polynomial
$g_D\in\F_q[T]$ such that
\[
 \Specper_q(\Psi_{g_D,e,n})=D.
\]
One may take the monic least common multiple
\[
 g_D(T)=\lcm_{d\in\operatorname{Max}(D)}(T^{q^d}-T),
\]
where $\operatorname{Max}(D)$ denotes the elements maximal under divisibility.
\end{theorem}

\begin{proof}
The roots of $g_D$ are the union of the fields $\F_{q^d}$ for
$d\in\operatorname{Max}(D)$. If $m\in D$, then $m$ divides some maximal
$d\in D$, so $\F_{q^m}\subseteq\F_{q^d}$ and
$T^{q^m}-T\mid g_D$.

Conversely, suppose $m\notin D$. Then $m$ divides no maximal element of
$D$, since $D$ is divisor-closed. Choose an element of degree exactly $m$
over $\F_q$; such an element exists because an irreducible polynomial of
every positive degree exists over a finite field
\cite[Theorem~3.25]{LidlNiederreiter1997}. It belongs to none of the fields
$\F_{q^d}$ with $d\in\operatorname{Max}(D)$, and hence it is not a root of
$g_D$. Therefore $T^{q^m}-T\nmid g_D$. The result follows from
Corollary~\ref{cor:spectrumcriterion}.
\end{proof}

\begin{remark}[Degree-minimal defining polynomial]\label{rem:gDminimal}
The polynomial $g_D$ is the unique monic polynomial of least degree among
all nonzero $g\in\F_q[T]$ satisfying $g(0)=0$ for which the spectrum of $\Psi_{g,e,n}$ contains $D$.
Indeed, if $\Specper_q(\Psi_{g,e,n})$ contains $D$, then
$T^{q^d}-T\mid g$ for every $d\in D$, and hence $g_D\mid g$. Consequently,
$\deg g\geq\deg g_D$, with equality for a monic $g$ only when $g=g_D$.
Since Theorem~\ref{thm:realization} shows that $g_D$ realizes $D$ exactly,
it is in particular degree-minimal among exact realizers. It is squarefree,
and its roots are exactly
$\bigcup_{d\in\operatorname{Max}(D)}\F_{q^d}$.
\end{remark}

\begin{corollary}[Complete spectrum classification in the family]
\label{cor:familyclassification}
Let $g\in\F_q[T]$ satisfy $g(0)=0$, and suppose that
$\Psi_{g,e,n}$ is a mock automorphism over $\F_q$. Then exactly one of the
following occurs:
\begin{enumerate}
\item $g=0$, in which case $\Psi_{g,e,n}=\id$ and
$\Specper_q(\Psi_{g,e,n})=\N_{>0}$;
\item $g\neq0$, in which case the spectrum is finite and
\[
 m\in\Specper_q(\Psi_{g,e,n})
 \quad\Longrightarrow\quad q^m\leq\deg g.
\]
\end{enumerate}
Conversely, $\N_{>0}$ and every finite nonempty divisor-closed subset of
$\N_{>0}$ occur as spectra of mock automorphisms in this family.
\end{corollary}

\begin{proof}
If $g=0$, the map is the identity. If $g\neq0$ and
$m\in\Specper_q(\Psi_{g,e,n})$, then $T^{q^m}-T$ divides $g$ by
Corollary~\ref{cor:spectrumcriterion}; hence $q^m\leq\deg g$. The converse
realization is Theorem~\ref{thm:realization}, together with the identity
map.
\end{proof}

\begin{example}\label{ex:spectrum}
Let $D=\{1,2,3,4,6\}$. Its maximal elements under divisibility are $4$ and
$6$, so Theorem~\ref{thm:realization} gives
\[
 g_D(T)=\lcm(T^{q^4}-T,T^{q^6}-T)
       =\frac{(T^{q^4}-T)(T^{q^6}-T)}{T^{q^2}-T},
\]
since $\gcd(T^{q^4}-T,T^{q^6}-T)=T^{q^2}-T$.
For every $e\geq1$ and $n\geq2$, the map $\Psi_{g_D,e,n}$ is bijective
exactly over the extensions of degrees $1,2,3,4$, and $6$.
\end{example}

\begin{corollary}[Uniform counterexample]\label{cor:uniformmulti}
For every finite field $\F_q$, every $e\geq1$, and every $n\geq2$, the map
\begin{equation}\label{eq:uniformmulti}
 F_{q,e,n}(X_1,\ldots,X_n)
 =\left(X_1+\bigl(X_2(X_1^q-X_1)\bigr)^{p^e},
 X_2,\ldots,X_n\right)
\end{equation}
is a mock automorphism with
\[
 \Specper_q(F_{q,e,n})=\{1\}.
\]
It is noninjective over every proper perfect overfield of $\F_q$.
\end{corollary}

\begin{proof}
Take $g(T)=T^q-T$. This polynomial is nonzero, has degree $q$, and already
has the $q$ distinct roots in $\F_q$. Hence its roots in every field
extension are exactly the elements of $\F_q$. Corollary~\ref{cor:properperfect}
and the finite-field criterion in Theorem~\ref{thm:exactcriterion} give the
assertions.
\end{proof}

Thus Maubach--Willems Conjecture~2.7 fails not merely because some mock automorphism has a
finite spectrum: every finite divisor-closed spectrum allowed by the formal
subfield constraint occurs in the family \eqref{eq:psi}.

\section{Geometry and a sharp normalized degree bound}
\label{sec:degree}

We first record the geometric degree and the source of nonproperness.

\begin{proposition}\label{prop:genericdegree}
Let $g\in\F_q[T]$ have degree $d\geq1$. The morphism
$\Psi_{g,e,n}\colon\A^n\to\A^n$ is étale and generically finite of degree
\[
 dp^e.
\]
It is not finite and hence not proper.
\end{proposition}

\begin{proof}
The Jacobian matrix is $I_n$, so the morphism is étale. Write
$Y_i=\Psi_{g,e,n,i}(X_1,\ldots,X_n)$ for the target coordinates and put
\[
 K=\F_q(X_2,\ldots,X_n),\qquad
 h(T)=T+X_2^{p^e}g(T)^{p^e}\in K[T].
\]
Since $Y_i=X_i$ for $i\geq2$, the induced extension of function fields is
\[
 K(h(X_1))\subseteq K(X_1).
\]
The polynomial $h$ has degree $dp^e$ and derivative $h'(T)=1$. Hence it is
separable and
\[
 [K(X_1):K(h(X_1))]=\deg h=dp^e;
\]
for instance, this follows by comparing the pole order of $h(X_1)$ at the
place at infinity of $K(X_1)$. The degree is unchanged after extending the
constant field to $\overline{\F}_q$, so this is the geometric generic
degree.

Suppose the morphism were finite. Being finite étale over the connected
target, it would have constant geometric fiber cardinality $dp^e$. Over a
target point with $Y_2=0$, however, the equations give $X_2=0$,
$X_i=Y_i$ for $i\geq3$, and $X_1=Y_1$. The fiber therefore consists of one
point, whereas $dp^e\geq p>1$. This is a contradiction. Since an étale
morphism is quasi-finite and a proper quasi-finite morphism is finite \cite[Tag~0F2P]{StacksProject}, nonfiniteness also implies
nonproperness.
\end{proof}

\begin{remark}
The étaleness, nonfiniteness, and variable fiber cardinalities in
Proposition~\ref{prop:genericdegree} are instances of the general geometry
of basic endomorphisms developed in
\cite{BorisovGabberVasiu2025}. The proposition records the exact generic
degree for the sparse family used in the arithmetic arguments below.
\end{remark}

The next theorem treats the entire strongly normalized class, rather than
only the separated family above.

\begin{definition}\label{def:normalizedclass}
Let $\mathcal N_{q,n}$ be the \,\emph{strongly normalized class} of maps
$F\in\F_q[X_1,\ldots,X_n]^n$ satisfying
\[
 \Jac(F)=I_n,
 \qquad F|_{\F_q^n}=\id_{\F_q^n}.
\]
\end{definition}

\begin{remark}\label{rem:normalizedclass}
The class $\mathcal N_{q,n}$ is stable under affine conjugation over
$\F_q$: if $A$ is affine, then $AFA^{-1}$ again has identity Jacobian
matrix and fixes $\F_q^n$ pointwise. The class is nevertheless restrictive.
An arbitrary mock automorphism cannot in general be transformed into it by
affine changes of coordinates, because neither a constant identity Jacobian
matrix nor the identity base-field function is forced by the mock condition.
\end{remark}

\begin{lemma}[Zero derivatives]\label{lem:zeroderiv}
Let $k$ be perfect of characteristic $p$. If
$H\in k[X_1,\ldots,X_n]$ satisfies
\[
 \frac{\partial H}{\partial X_j}=0
 \quad(1\leq j\leq n),
\]
then $H=G^p$ for some $G\in k[X_1,\ldots,X_n]$.
\end{lemma}

\begin{proof}
Every exponent occurring in $H$ is divisible by $p$. Thus
\[
 H\in k[X_1^p,\ldots,X_n^p].
\]
Since $k$ is perfect, every coefficient has a $p$th root in $k$.
\end{proof}

\begin{lemma}[The first layer of the vanishing ideal]\label{lem:vanishing}
If $H\in\F_q[X_1,\ldots,X_n]$ vanishes on $\F_q^n$ and
$\deg H\leq q$, then
\[
 H=\sum_{j=1}^n c_j(X_j^q-X_j)
\]
for some $c_j\in\F_q$.
\end{lemma}

\begin{proof}
The vanishing ideal of $\F_q^n$ is
\[
 I_q=(X_1^q-X_1,\ldots,X_n^q-X_n).
\]
Fix a degree-compatible monomial order. Then
\[
 \operatorname{LM}(X_j^q-X_j)=X_j^q,
\]
and these leading monomials are pairwise coprime, so the displayed
generators form a Gröbner basis. Divide $H$ by them. The remainder has
degree less than $q$ in every variable and represents the zero function on
$\F_q^n$; uniqueness of reduced polynomial representatives gives zero
remainder. Every quotient term has total degree at most
$\deg H-q\leq0$. Hence every quotient is constant, which proves the claim.
\end{proof}

\begin{lemma}[Additive maps have infinitely many permutation extensions]
\label{lem:additiveexceptional}
Let $F\colon\A^n_{\F_q}\to\A^n_{\F_q}$ be additive, by which we mean that
\[
 F(U+V)=F(U)+F(V)
\]
as a polynomial identity in two $n$-tuples of variables. Equivalently, $F$
is a homomorphism of the additive group scheme $\mathbf G_{a,\F_q}^n$;
on the points of every extension field it induces an $\F_p$-linear map.
Assume $\Jac(F)=I_n$ and $F|_{\F_q^n}=\id$. Then $\Specper_q(F)$ is infinite.
\end{lemma}

\begin{proof}
Because $F$ is étale, its geometric kernel
$\ker(F)(\overline{\F}_q)$ is finite. It contains no nonzero
$\F_q$-point. If the nonzero kernel is empty, then $F$ is injective, and
hence bijective, on every finite extension. Otherwise, let
$d_1,\ldots,d_s>1$ be the degrees of the minimal fields of definition of
its nonzero points. Every prime integer $m>\max_i d_i$ is divisible by none
of the $d_i$, so
\[
 \ker(F)(\F_{q^m})=\{0\}.
\]
For all such primes, the additive endomorphism of the finite group
$\F_{q^m}^n$ is injective and hence bijective.
\end{proof}

\begin{theorem}[Sharp normalized degree threshold]\label{thm:normalizeddegree}
Let $n\geq1$. If $F\in\mathcal N_{q,n}$ and
\[
 \deg F<p(q+1),
\]
then $F$ is extension-exceptional. Equivalently, every map in
$\mathcal N_{q,n}$ with finite spectrum has degree at least $p(q+1)$.
For every $n\geq2$, the
bound is attained by $F_{q,1,n}$ in \eqref{eq:uniformmulti}, whose spectrum
is $\{1\}$.
\end{theorem}

\begin{proof}
Write $F_i=X_i+H_i$. Since $\Jac(F)=I_n$, all partial derivatives of each
$H_i$ vanish. By Lemma~\ref{lem:zeroderiv}, $H_i=G_i^p$. The identity on
$\F_q^n$ implies that $G_i$ vanishes on $\F_q^n$.

For $G_i\neq0$, one has $\deg(G_i^p)=p\deg G_i$; the case $G_i=0$ is
harmless. The degree hypothesis therefore gives $\deg G_i<q+1$, hence
$\deg G_i\leq q$.
Lemma~\ref{lem:vanishing} yields
\[
 G_i=\sum_{j=1}^n c_{ij}(X_j^q-X_j).
\]
Therefore
\[
 F_i=X_i+\sum_{j=1}^n c_{ij}^p(X_j^{pq}-X_j^p),
\]
so $F$ satisfies the additive polynomial identity. Lemma~\ref{lem:additiveexceptional} applies.
The final assertion follows from Corollary~\ref{cor:uniformmulti}, whose map
has degree exactly $p(q+1)$.
\end{proof}

\begin{remark}\label{rem:globaldegree}
Theorem~\ref{thm:normalizeddegree} is global inside the strongly normalized
class $\mathcal N_{q,n}$; it is not merely a statement about separated maps.
It does not determine the least degree among all mock automorphisms. Affine
normalization can arrange $F(0)=0$ and $\Jac(F)(0)=I_n$, but it cannot in
general force the polynomial identities $\Jac(F)=I_n$ and
$F|_{\F_q^n}=\id$. An unrestricted minimum would require new control of
low-degree Keller maps outside this class.
\end{remark}

\section{One-variable counterexamples}
\label{sec:onevariable}

A one-variable Keller map is simply a polynomial with nonzero constant
derivative. We first construct counterexamples with exact spectrum $\{1\}$
over every odd finite field.

\subsection{Exact spectrum over odd finite fields}

Assume $q$ is odd and set
\begin{equation}\label{eq:uq}
 u_q(T)=T+\bigl(T(T^q-T)\bigr)^p\in\F_q[T].
\end{equation}
Then $u_q'(T)=1$ and $u_q(a)=a$ for $a\in\F_q$.

\begin{lemma}[A norm-one choice]\label{lem:normchoice}
Let $K=\F_{q^m}$ with $q$ odd and $m>1$, and let
$\Norm=\Norm_{K/\F_q}$. There is $z\in K^\times$ such that
\[
 \Norm(z)=1,
 \qquad \Norm(2-z)\neq1.
\]
\end{lemma}

\begin{proof}
Let
\[
 M=\frac{q^m-1}{q-1},
 \qquad H=\ker(\Norm)\subseteq K^\times.
\]
Thus $|H|=M$, and $H$ is the set of roots of the separable polynomial
$Z^M-1$ because $p\nmid M$. Suppose, toward a contradiction,
that $\Norm(2-z)=1$ for every $z\in H$. Then
\[
 P(Z)=\prod_{i=0}^{m-1}(2-Z^{q^i})-1
\]
vanishes on $H$. Since $\deg P=M$, it follows that
\[
 P(Z)=(-1)^m(Z^M-1).
\]
However, because $m>1$, the coefficient of $Z^{M-1}$ on the left is
\[
 2(-1)^{m-1}\neq0.
\]
Indeed, the exponents obtained by choosing either $2$ or
$-Z^{q^i}$ from each factor have base-$q$ digits in $\{0,1\}$. Since
$q>2$, uniqueness of the base-$q$ expansion shows that the only way to
obtain
\[
 M-1=q+q^2+\cdots+q^{m-1}
\]
is to choose the constant term $2$ from the factor indexed by $i=0$ and the
monomial term from every other factor. The coefficient of $Z^{M-1}$ on the
right is zero, a contradiction.
\end{proof}

\begin{theorem}[Exact odd-characteristic spectrum]\label{thm:oddexact}
For every odd prime power $q$,
\[
 \Specper_q(u_q)=\{1\}.
\]
Equivalently, $u_q$ fixes $\F_q$ pointwise and is not injective on
$\F_{q^m}$ for any $m>1$.
\end{theorem}

\begin{proof}
Fix $K=\F_{q^m}$ with $m>1$. Choose $z$ as in
Lemma~\ref{lem:normchoice}. The image of the map
$K^\times\to K^\times$, $s\mapsto s^{q-1}$, is exactly
$\ker(\Norm)$; this is the standard norm criterion in a cyclic finite-field
extension, see \cite[Section~2.3]{LidlNiederreiter1997}. Choose
$s\in K^\times$ satisfying
\[
 s^{q-1}=z.
\]
Let $w\in K$ be the unique element with $w^p=-s$.

Put $G(T)=T(T^q-T)$. For $x=y+s$, a direct calculation gives
\begin{equation}\label{eq:Gdifference}
 G(y+s)-G(y)
 =s\bigl(y^q+(z-2)y+s^q-s\bigr).
\end{equation}
Consider the $\F_q$-linear map
\[
 L_z(y)=y^q+(z-2)y.
\]
It has a nonzero kernel precisely when $2-z$ is a $(q-1)$st power in $K$.
For a nonzero element, this is equivalent to norm one. When $2-z=0$,
the equation defining the kernel is $y^q=0$, so the kernel is zero. Our choice of $z$ therefore makes $L_z$ invertible.
Choose $y\in K$ solving
\[
 L_z(y)=\frac{w}{s}-s^q+s.
\]
By \eqref{eq:Gdifference}, $G(y+s)-G(y)=w$. Hence
\[
 u_q(y+s)-u_q(y)
 =s+w^p=0.
\]
Since $s\neq0$, this is a collision between distinct elements.
\end{proof}

\begin{remark}
The proof is genuinely uniform in the extension degree. It gives more than
nonexceptionality: no proper finite extension belongs to the spectrum.
\end{remark}

\subsection{The field of two elements}

For $q=2$, a particularly sparse exact example is
\begin{equation}\label{eq:f2}
 u_2(T)=T+T^2+T^6.
\end{equation}
It satisfies $u_2'=1$ and fixes $\F_2$ pointwise.

\begin{theorem}[Exact binary spectrum]\label{thm:f2exact}
For $u_2$ in \eqref{eq:f2},
\[
 \Specper_2(u_2)=\{1\}.
\]
\end{theorem}

\begin{proof}
Let $K=\F_{2^m}$ with $m>1$, and write $\Tr=\Tr_{K/\F_2}$. We construct
distinct $x,y\in K$ with $u_2(x)=u_2(y)$.

Put $s=x+y$ and $t=xy$. In characteristic two,
\[
 x^2+y^2=s^2,
 \qquad
 x^6+y^6=(x^3+y^3)^2=(s^3+st)^2=s^6+s^2t^2.
\]
Thus a collision with $s\neq0$ is equivalent to
\begin{equation}\label{eq:tcondition}
 t^2=s^4+1+s^{-1}.
\end{equation}
Let $r=s^{-1}$ and $b=t/s^2$. Then \eqref{eq:tcondition} becomes
\[
 b^2=1+r^4+r^5,
\]
and therefore
\begin{equation}\label{eq:traceb}
 \Tr(b)=\Tr(b^2)=\Tr(1+r+r^5).
\end{equation}
We choose $r\neq0$ so that the last trace is zero.

If $m$ is even, take $r=1$. Then
$\Tr(1+r+r^5)=\Tr(1)=0$. Suppose now that $m$ is odd, so $m\geq3$, and define the quadratic map
\[
 Q(r)=\Tr(r^5+r).
\]
Here ``quadratic'' is meant in the characteristic-two sense: the linear term
is allowed, and the polar form below is bilinear. Its polar form is
\[
 B(r,w)=\Tr(r^4w+rw^4)
 =\Tr\bigl(r(w^{2^{m-2}}+w^4)\bigr).
\]
If $B$ were identically zero, nondegeneracy of the trace pairing would give
$w^{2^{m-2}}=w^4$ for every $w\in K$. The corresponding Frobenius
endomorphisms would be equal, forcing $m-2\equiv2\pmod m$ and hence
$m\mid4$, impossible for odd $m\geq3$. Thus $Q$ is not the zero function. Since $Q(0)=0$, there is
$r\neq0$ with $Q(r)=1$. Since $\Tr(1)=1$ for odd $m$, this gives
$\Tr(1+r+r^5)=0$.

Having chosen $r$, put $s=r^{-1}$. The squaring map is bijective on $K$, so
choose $t$ satisfying \eqref{eq:tcondition}; then $b=t/s^2$ has trace zero by
\eqref{eq:traceb}. The Artin--Schreier equation
\[
 v^2+v=b
\]
therefore has a solution $v\in K$. Indeed, the $\F_2$-linear map
$v\mapsto v^2+v$ has kernel $\F_2$, hence image of cardinality
$2^{m-1}$; its image is contained in the trace-zero hyperplane, which has
the same cardinality. Set
\[
 x=sv,
 \qquad y=s(v+1).
\]
Then $x+y=s\neq0$ and $xy=s^2b=t$, so the preceding calculation gives
$u_2(x)=u_2(y)$.
\end{proof}

\subsection{An elementary construction over every finite field}

The preceding two theorems cover all odd fields and $\F_2$ with exact
spectrum. The following construction gives a short elementary proof of
nonexceptionality for every remaining field as well.

For $q>2$, let $\ell_q$ denote the least prime divisor of $q-1$ and set
\begin{equation}\label{eq:vq}
 v_q(T)=T+\bigl(T^{\ell_q-1}(T^q-T)\bigr)^p.
\end{equation}

\begin{theorem}[Elementary uniform one-variable counterexamples]
\label{thm:clwfamily}
For every $q>2$, the polynomial $v_q$ is a mock automorphism over $\F_q$
and is nonexceptional. Its degree is
\[
 d_q=p(q+\ell_q-1),
\]
and its spectrum satisfies
\[
 \Specper_q(v_q)
 \subseteq\{m\geq1:q^m<d_q^4\}.
\]
For odd $q$, one has $\ell_q=2$, so $v_q=u_q$ and
$\Specper_q(v_q)=\{1\}$ by Theorem~\ref{thm:oddexact}.
\end{theorem}

\begin{proof}
The derivative of $v_q$ is one, and $v_q(a)=a$ for every $a\in\F_q$.
Moreover,
\[
 \ell_q\mid(q-1)
 \quad\text{and}\quad
 \ell_q\mid(q+\ell_q-1),
\]
so
\[
 \gcd(d_q,q-1)\geq\ell_q>1.
\]
Theorem~\ref{thm:clw} shows that $v_q$ is nonexceptional. The spectrum bound
is Proposition~\ref{prop:effective}.
\end{proof}

\begin{remark}[The remaining even-field spectrum problem]
\label{rem:evenunknown}
When $q>2$ is even, Theorem~\ref{thm:clwfamily} proves that the spectrum is
finite and reduces its exact determination to finitely many extension
fields. We do not prove that it is $\{1\}$. Closing this point would require
either a collision construction valid in every extension, analogous to the
norm argument in odd characteristic, or an explicit factorization and
Frobenius analysis of the off-diagonal curve. The Carlitz--Lenstra--Wan
degree obstruction alone cannot determine the individual extension degrees.
\end{remark}

\subsection{Pointwise-identity degree optimality in one variable}

Let
\[
 \nu_q=
 \min\{\deg f:f\in\F_q[T],\ f'=1,\ f|_{\F_q}=\id,
                    \text{ and }f\text{ is nonexceptional}\}.
\]
This minimum exists by Theorems~\ref{thm:f2exact}, \ref{thm:oddexact}, and
\ref{thm:clwfamily}.

\begin{proposition}\label{prop:onepointwise}
One has
\[
 \nu_q\geq p(q+1).
\]
Equality holds for every odd $q$ and for $q=2$. For even $q>2$, the present
paper gives
\[
 p(q+1)\leq\nu_q\leq p(q+\ell_q-1).
\]
\end{proposition}

\begin{proof}
Since $f'-1=0$, Lemma~\ref{lem:zeroderiv} in one variable gives
\[
 f(T)=T+G(T)^p
\]
for some $G\in\F_q[T]$. The identity condition implies that $G$ vanishes on
$\F_q$, so $T^q-T$ divides $G$. If $\deg f<p(q+1)$, then
$\deg G\leq q$, whence $G=c(T^q-T)$. Thus $f$ is affine linearized. If it
permutes $\F_q$, the finite-kernel argument in
Lemma~\ref{lem:additiveexceptional} shows that it permutes infinitely many
extensions. Therefore a nonexceptional $f$ must have degree at least
$p(q+1)$. The equality examples follow from
Theorems~\ref{thm:oddexact} and \ref{thm:f2exact}.
\end{proof}

\section{Prime fields and exact least degrees}
\label{sec:minimal}

We now remove the pointwise-identity restriction and study the least possible
degree over $\F_p$.

\begin{definition}\label{def:delta}
For every prime $p$, let $\mu_p$ be the least degree of a nonexceptional
one-variable mock automorphism over $\F_p$. For $p\geq5$, let $\delta_p$ be
the least degree of a nonlinear reduced permutation polynomial over $\F_p$,
where ``reduced'' means the unique representative modulo $T^p-T$ of degree
at most $p-1$.
\end{definition}

The number $\delta_p$ exists because $T^{p-2}$ permutes $\F_p$ for $p\geq5$.
The argument in this section has three layers. Proposition~\ref{prop:mulower}
gives an elementary lower bound. Corollary~\ref{cor:classificationfree} and
Proposition~\ref{prop:cubiccurve} establish equality in broad families by
direct arguments. The uniform equality in Theorem~\ref{thm:muexact} uses the
classification theorem quoted as Theorem~\ref{thm:fgsdegrees}; the final
congruence calculations are applications of existing low-degree
permutation-polynomial classifications.

We begin with two elementary lemmas.

\begin{lemma}[Constant derivative and decomposition]\label{lem:decomp}
Let $k$ have characteristic $p$, and let $f\in k[T]$ have nonzero constant
derivative. If
\[
 f=A\circ B
\]
with $\deg A,\deg B>1$, then $p$ divides both $\deg A$ and $\deg B$. In
particular, if $p^2\nmid\deg f$, then $f$ is absolutely indecomposable.
\end{lemma}

\begin{proof}
The chain rule gives
\[
 f'=(A'\circ B)B'.
\]
A product of two polynomials is a nonzero constant only if both factors are
nonzero constants. Thus $B'$ is constant, and $A'\circ B$ constant implies
that $A'$ is constant. A nonlinear polynomial with constant derivative in
characteristic $p$ has degree divisible by $p$. The argument is unchanged
after extending the coefficient field.
\end{proof}

\begin{lemma}[Linearized base permutations are exceptional]
\label{lem:linearizedexceptional}
Let
\[
 f(T)=a_0+a_1T+a_pT^p+\cdots+a_{p^r}T^{p^r}\in\F_p[T]
\]
be affine linearized. If $f$ permutes $\F_p$, then it permutes
$\F_{p^m}$ for infinitely many $m$.
\end{lemma}

\begin{proof}
Translations do not affect injectivity, so consider the additive part $L$.
Since $f$ permutes $\F_p$, the polynomial $L$ is nonzero. Hence its set of
roots in $\overline{\F}_p$ is finite, with no separability hypothesis
needed. No nonzero root lies in $\F_p$, because $L$ is injective there.
If there are no nonzero geometric roots, then $L$ is injective on every
finite extension. Otherwise, let $d_1,\ldots,d_s>1$ be the degrees of the
minimal fields of definition of the nonzero roots. Every prime integer
$m>\max_i d_i$ is divisible by none of the $d_i$, so $L$ has zero kernel on
$\F_{p^m}$ and is therefore bijective on that finite field.
\end{proof}

The following result gives a direct nonexceptionality proof for the cubic
lift used in the sharpest infinite family of prime-field examples.

\begin{proposition}[An explicit collision curve]\label{prop:cubiccurve}
Let $p\geq5$ and
\[
 c_p(T)=T+(T^3-T)^p\in\F_p[T].
\]
Then $c_p$ is nonexceptional. Moreover, it induces the function
$T\mapsto T^3$ on $\F_p$; hence it is a mock automorphism precisely when
$p\equiv2\pmod3$.
\end{proposition}

\begin{proof}
For $a\in\F_p$ one has
\[
 c_p(a)=a+(a^3-a)^p=a^3,
\]
which proves the assertion about the base-field function. It remains to
prove nonexceptionality.

Consider the affine curve $C$ over $\F_p$ with coordinates $(t,y)$ defined
by
\begin{equation}\label{eq:cubiccurve}
 Q(t,y):=
 3t^{p-1}y^2+3t^{2p-1}y+t^{3p-1}-t^{p-1}+1=0.
\end{equation}
As a quadratic in $y$ over $\overline{\F}_p(t)$, its discriminant is
\[
 -3t^{p-1}R(t),
 \qquad R(t)=t^{3p-1}-4t^{p-1}+4.
\]
The factor $t^{p-1}$ is a square. Moreover,
\[
 R'(t)=t^{p-2}(4-t^{2p}).
\]
If a nonzero root of $R$ were also a root of $R'$, then $t^{2p}=4$ and
\[
 R(t)=t^{p-1}(t^{2p}-4)+4=4,
\]
a contradiction; also $R(0)=4$. Hence $R$ is squarefree and nonconstant,
so the discriminant is not a square in $\overline{\F}_p(t)$. The
quadratic polynomial $Q$ is primitive in $\overline{\F}_p[t][y]$, since its
constant coefficient is congruent to $1$ modulo $t$. It is therefore
irreducible over $\overline{\F}_p(t)$ by the discriminant criterion and
absolutely irreducible in $\overline{\F}_p[t,y]$ by Gauss's lemma.

Set $x=y+t^p$. Multiplying \eqref{eq:cubiccurve} by $t$ gives
\[
 (y+t^p)^3-y^3-t^p+t=0,
\]
so
\[
 (x^3-x)-(y^3-y)=-t.
\]
Raising to the $p$th power and using $x-y=t^p$ yields
\[
 c_p(x)-c_p(y)=0.
\]
We now justify scheme-theoretically that this parametrizes an absolutely
irreducible off-diagonal component.

Let
\[
 B=\F_p[t,y]/(Q),
\]
and consider the homomorphism
\[
 \varphi:\F_p[X,Y]\longrightarrow B,
 \qquad X\longmapsto y+t^p,\qquad Y\longmapsto y.
\]
Put
\[
 A=\operatorname{im}(\varphi)=\F_p[y,y+t^p]\subseteq B,
 \qquad P=\ker(\varphi).
\]
Then $\operatorname{Spec}A$ is the scheme-theoretic image of $C$ in
$\A^2_{X,Y}$. Write
\[
 S=3y^2+3t^py+t^{2p}-1\in A.
\]
The equation $Q=0$ gives $1=-t^{p-1}S$ in $B$. Multiplying by $t$ yields
\[
 t=-t^pS\in A,
\]
because $t^p=X-Y\in A$. Hence $A=B$. In particular, $A$ is a
domain of dimension one, and for every extension $K/\F_p$ the ring
$A\otimes_{\F_p}K=B\otimes_{\F_p}K$ is a domain because $Q$ is absolutely
irreducible. Thus the image is a closed geometrically integral curve.
Moreover, $t$ is a unit in $A$, with inverse $-t^{p-2}S$, so
$X-Y=t^p$ is a unit on the image.

The prime ideal $P$ has height one in the UFD $\F_p[X,Y]$, and hence
$P=(H)$ for an irreducible polynomial $H\in\F_p[X,Y]$. Geometric
integrality of $A$ implies that $H$ is absolutely irreducible. The preceding
collision identity shows that
\[
 c_p(X)-c_p(Y)\in P,
\]
so $H$ divides $c_p(X)-c_p(Y)$. Since $X-Y$ is a unit modulo $P$, the
polynomial $H$ is not associated with $X-Y$. Therefore
\[
 c_p(X)-c_p(Y)=(X-Y)\Phi_{c_p}(X,Y)
\]
implies $H\mid\Phi_{c_p}(X,Y)$. Finally, $c_p'(T)=1$, so the standard
separable collision-polynomial criterion applies. Consequently $c_p$ is
nonexceptional.
\end{proof}

We first isolate the elementary lower bound.

\begin{proposition}[Elementary lower bounds]\label{prop:mulower}
One has
\[
 \mu_2\geq6,\qquad \mu_3\geq12,
 \qquad \mu_p\geq p\delta_p\quad(p\geq5).
\]
\end{proposition}

\begin{proof}
Let first $p\geq5$, and let $f$ be a mock automorphism of degree less than
$p\delta_p$. Write $f'=c\in\F_p^\times$. By
Proposition~\ref{prop:invariance}, replacing $f$ by
$c^{-1}(f-f(0))$ preserves the spectrum and the degree; hence we may assume
$f'=1$ and $f(0)=0$. Then
\[
 f(T)=T+H(T)^p
\]
for some $H\in\F_p[T]$. The degree bound gives $\deg H<\delta_p<p$. On
$\F_p$, the polynomial $f$ induces the same function as $T+H(T)$. Since this
is a reduced permutation polynomial of degree less than $\delta_p$, it must
be affine. Hence $H$ is affine, and $f$ is affine linearized. By
Lemma~\ref{lem:linearizedexceptional}, it is exceptional.

For $p=2$, after normalizing the derivative, every polynomial of degree less
than six is of the form $T+H(T)^2$ with $\deg H\leq2$. Writing
$H=aT^2+bT+c$ shows explicitly that
\[
 T+H(T)^2=T+a^2T^4+b^2T^2+c^2,
\]
so the polynomial is affine linearized and exceptional by
Lemma~\ref{lem:linearizedexceptional}.
For $p=3$, a normalized polynomial of degree less than twelve is
$T+H(T)^3$ with $\deg H\leq3$. Modulo $T^3-T$, the base-field function
$T+H(T)$ has degree at most two. A quadratic polynomial over an odd field
cannot be a permutation, so the reduced function is affine. If
$H=aT^3+bT^2+cT+d$, this forces $b=0$; hence
$T+H(T)^3$ is again affine linearized and exceptional.
\end{proof}

There are two broad situations in which equality follows without any
classification of exceptional polynomials.

\begin{corollary}[Classification-free equality cases]
\label{cor:classificationfree}
Let $p\geq5$.
\begin{enumerate}
\item If $p\equiv2\pmod3$, then $\mu_p=3p$.
\item If $\gcd(\delta_p,p-1)>1$, then $\mu_p=p\delta_p$.
\end{enumerate}
\end{corollary}

\begin{proof}
In the first case, no quadratic polynomial permutes an odd field and
$T^3$ is a permutation of $\F_p$. Hence $\delta_p=3$, and
Proposition~\ref{prop:cubiccurve} gives a degree-$3p$ nonexceptional mock
automorphism. In the second case, let $P$ be a reduced permutation polynomial
of degree $\delta_p$ and set
\begin{equation}\label{eq:liftP}
 f_P(T)=T+(P(T)-T)^p.
\end{equation}
This polynomial has derivative one, induces $P$ on $\F_p$, and has degree
$p\delta_p$. Since
\[
 \gcd(p\delta_p,p-1)=\gcd(\delta_p,p-1)>1,
\]
Theorem~\ref{thm:clw} shows directly that $f_P$ is nonexceptional. The lower
bounds come from Proposition~\ref{prop:mulower}.
\end{proof}

For the remaining cases we use the following precise consequence of the
Fried--Guralnick--Saxl classification.

\begin{theorem}[Degree alternatives for indecomposable exceptional
polynomials]\label{thm:fgsdegrees}
Let $k$ be a finite field of characteristic $p$, and let $f\in k[T]$ be
separable, indecomposable over $k$, and exceptional of degree $d>1$. Then
one of the following holds:
\begin{enumerate}
\item $d$ is prime and $d\neq p$;
\item $d=p^e$ for some $e\geq1$;
\item $p\in\{2,3\}$ and
$d=p^e(p^e-1)/2$ for some odd $e>1$.
\end{enumerate}
In particular, when $p\geq5$, only the first two alternatives occur.
\end{theorem}

\begin{proof}[Reference]
This is the degree part of the Fried--Guralnick--Saxl theorem
\cite{FriedGuralnickSaxl1993}, stated with its hypotheses in
\cite[Theorem~1.1]{GuralnickRosenbergZieve2010}. The theorem concerns
indecomposability over the ground field; absolute indecomposability, which
is what Lemma~\ref{lem:decomp} yields below, is stronger.
\end{proof}

\begin{theorem}[Classification-dependent exact prime-field minimum]\label{thm:muexact}
The least degree $\mu_p$ of a nonexceptional one-variable mock automorphism
over $\F_p$ is
\[
 \mu_2=6,\qquad \mu_3=12,\qquad
 \mu_p=p\delta_p\quad(p\geq5).
\]
For $p\geq5$, if $P$ is any reduced nonlinear permutation polynomial of
least degree $\delta_p$, then the polynomial $f_P$ in
\eqref{eq:liftP} is a degree-$p\delta_p$ nonexceptional mock automorphism.
\end{theorem}

\begin{proof}
The lower bounds are Proposition~\ref{prop:mulower}. The polynomials $u_2$
and $u_3$ give equality for $p=2$ and $p=3$ by
Theorems~\ref{thm:f2exact} and \ref{thm:oddexact}.

Let $p\geq5$, let $P$ have degree $\delta_p$, and consider $f_P$. It has
derivative one and induces $P$ on $\F_p$, so it is a mock automorphism.
Its degree is $p\delta_p$. Since $P$ is nonlinear and reduced,
$1<\delta_p<p$, and therefore $p\nmid\delta_p$. In particular,
$p^2\nmid p\delta_p$. By Lemma~\ref{lem:decomp}, $f_P$ is absolutely
indecomposable. If it were exceptional,
Theorem~\ref{thm:fgsdegrees} would force its degree to be prime or a power
of $p$. The integer $p\delta_p$ is composite and, because
$p\nmid\delta_p$, is not a power of $p$. Thus $f_P$ is nonexceptional.
\end{proof}

\begin{remark}[Where the deep input remains]
The principal existence theorem, Theorem~\ref{thm:clwfamily}, uses only the
Carlitz--Lenstra--Wan degree obstruction. Corollary~\ref{cor:classificationfree}
removes the exceptional-polynomial classification whenever
$\gcd(\delta_p,p-1)>1$, and Proposition~\ref{prop:cubiccurve} removes it for
the infinite family $p\equiv2\pmod3$. The classification theorem is retained
only for the residual cases needed to prove $\mu_p=p\delta_p$ uniformly for
all $p\geq5$. Eliminating it there would require a direct absolutely
irreducible collision component for every lift \eqref{eq:liftP}, or another
general obstruction to exceptionality.
\end{remark}

We next record exactly the low-degree inputs used in the concrete
evaluations. Dickson's classification of permutation polynomials of degree at
most five is reproduced in the published \emph{Handbook of Finite Fields}
\cite[Section~8.1, Remark~8.1.8]{MullenWang2013}; see also
\cite[Introduction]{ShallueWanless2013} for a modern confirmation of the
scope of Dickson's list. For odd prime fields, the relevant normalized
consequences are as follows: cubics are represented by $T^3$ outside
characteristic three; quartics occur
only over $\F_7$ (apart from characteristic-two families); and the degree-five
list contains $T^5$ together with additional nonmonomial families. In
particular, when $p\equiv\pm2\pmod5$ one such family is
\[
 T^5+aT^3+5^{-1}a^2T,
\]
where $a\in\F_p$ is arbitrary and $5^{-1}$ denotes the inverse of $5$ in
$\F_p$. Inspection of the complete list shows that none of its degree-five
families occurs for $p>7$ with $p\equiv1\pmod5$.

\begin{lemma}[Low-degree permutation-polynomial consequences]
\label{lem:lowdegreepp}
Let $p$ be an odd prime.
\begin{enumerate}
\item No polynomial of degree two permutes $\F_p$.
\item If $p\geq5$, there exists a permutation polynomial of degree three
over $\F_p$ if and only if $p\not\equiv1\pmod3$.
\item If $p>7$, there is no permutation polynomial of degree four over
$\F_p$.
\item If $p>7$ and $p\equiv1\pmod5$, there is no permutation polynomial of
degree five over $\F_p$.
\end{enumerate}
Moreover, $T^4+3T$ permutes $\F_7$.
\end{lemma}

\begin{proof}[Reference and verification]
The published normalized list in
\cite[Section~8.1, Remark~8.1.8]{MullenWang2013} gives items (2)--(4).
In characteristic different from three, its cubic entry is $T^3$, which is a
permutation precisely when $\gcd(3,p-1)=1$, equivalently
$p\not\equiv1\pmod3$. The quartic entries over odd fields occur only over
$\F_7$, so no odd prime $p>7$ admits a quartic permutation polynomial.
For $p>7$, inspection of the complete degree-five list shows that none of
its entries applies when $p\equiv1\pmod5$; this proves item~(4).
Item (1) is elementary: after completing the square, a quadratic polynomial
over an odd field takes the same value at two distinct points. For the final
assertion, the values of $T^4+3T$ at $0,1,\ldots,6$ are respectively
\[
 0,4,1,6,2,3,5.
\]
\end{proof}

\begin{lemma}[A Hermite-criterion consequence]\label{lem:hermitedivisor}
Let $p$ be prime and let $1<d<p$. If $d\mid p-1$, then no polynomial of
degree $d$ permutes $\F_p$.
\end{lemma}

\begin{proof}
Hermite's criterion states that, for a permutation polynomial $f$ of degree
less than $p$, the reduction of $f^k$ modulo $T^p-T$ has degree at most
$p-2$ for every $1\leq k\leq p-2$; see
\cite[Lemma~2.1]{Fan2020}. Take $k=(p-1)/d$. The leading term of $f^k$ has
nonzero coefficient and degree $p-1$, while every other term has smaller
degree. Since $p-1<p$, reduction modulo $T^p-T$ does not alter this term,
contradicting Hermite's criterion.
\end{proof}

\begin{corollary}[Evaluations of $\mu_p$]\label{cor:muclasses}
For primes $p\geq5$:
\begin{enumerate}
\item if $p\equiv2\pmod3$, then $\delta_p=3$ and $\mu_p=3p$;
\item $\delta_7=4$ and $\mu_7=28$;
\item if $p>7$, $p\equiv1\pmod3$, and $p\not\equiv1\pmod5$, then
$\delta_p=5$ and $\mu_p=5p$;
\item if $p\equiv1\pmod{15}$ but $p\not\equiv1\pmod7$, then
$\delta_p=7$ and $\mu_p=7p$;
\item if $p\equiv1\pmod{105}$, then $\delta_p\geq9$ and
$\mu_p\geq9p$.
\end{enumerate}
\end{corollary}

\begin{proof}
By Lemma~\ref{lem:lowdegreepp}(1)--(2), if $p\equiv2\pmod3$ then
$\delta_p=3$, proving (1). For $p=7$, the same lemma excludes degrees two
and three and exhibits the degree-four permutation $T^4+3T$, proving (2).

For (3), Lemma~\ref{lem:lowdegreepp} excludes degrees two, three, and four,
while $T^5$ is a permutation because $\gcd(5,p-1)=1$. If
$p\equiv1\pmod{15}$, then the oddness of $p$ implies
$p\equiv1\pmod{30}$, so $6\mid p-1$. Lemma~\ref{lem:lowdegreepp}
excludes degrees at most five, and Lemma~\ref{lem:hermitedivisor} excludes
degree six. If
$p\not\equiv1\pmod7$, then $T^7$ is a permutation, proving (4).

Finally, if $p\equiv1\pmod{105}$, the same arguments exclude degrees at
most seven. Such a prime is greater than $31$. Fan's classification states
that, for odd $q>8$, degree-eight permutation polynomials exist exactly for
$q\in\{11,13,19,23,27,29,31\}$
\cite[Theorem~1.1]{Fan2020}; hence degree eight is also impossible. This
proves (5). The assertions for $\mu_p$ follow from
Theorem~\ref{thm:muexact}.
\end{proof}

\begin{remark}[The residual least-degree problem]
Corollary~\ref{cor:muclasses} resolves several infinite congruence families,
but it does not provide a closed formula for $\delta_p$. For primes outside
the displayed cases, determining $\delta_p$ remains a low-degree
permutation-polynomial classification problem. Further degree-by-degree
classifications may yield additional cases, but they do not presently give
a uniform formula for all primes.
\end{remark}

\section{Mechanisms, repairs, and open questions}
\label{sec:repairs}

\subsection{Why the counterexamples work}

The counterexamples combine three independent facts. First, in
characteristic $p$,
\[
 d(H^{p^e})=0,
\]
so a nonlinear perturbation can be invisible to the Jacobian matrix. Second,
a factor such as $T^q-T$ makes the perturbation vanish on the base field.
Third, after extending the field, Frobenius surjectivity and the newly
available elements can be used to create a collision. In the multivariate
family, the free coordinate $X_2$ performs this tuning. In the odd
one-variable family, the norm-one subgroup and the invertible linearized
operator $y^q+(z-2)y$ replace that free coordinate.

All nontrivial counterexamples constructed in Sections~\ref{sec:frob} and
\ref{sec:onevariable} are étale, so inseparability is not the issue. Their
geometric generic degrees are divisible by $p$. In the multivariate family,
nonproperness is visible in the jump from the generic fiber cardinality to
the one-point fiber above $Y_2=0$. These features are consistent with the
broader theory of basic endomorphisms in
\cite{BorisovGabberVasiu2025}. What is special to the present paper is the
exact finite-extension behavior.

The collision-scheme formulation \eqref{eq:collisioncriterion} also
clarifies the relation with exceptional covers. For a finite cover, the
components of the off-diagonal fiber product are controlled by monodromy and
constant-field extensions. Our multivariate examples are generally
nonfinite, so that framework does not directly constrain their spectra. The
one-variable examples are finite and therefore show that nonfiniteness is
not the only obstruction.

\subsection{Finiteness does not repair the conjecture}

\begin{proposition}\label{prop:finitefails}
The finite-morphism version of Maubach--Willems Conjecture~2.7 is false.
Indeed, every one-variable counterexample in Section~\ref{sec:onevariable}
is a finite étale morphism $\A^1\to\A^1$.
\end{proposition}

\begin{proof}
Every nonconstant polynomial map of the affine line is finite. The derivative
of each displayed polynomial is one, so each morphism is étale. Its
extension spectrum is finite, and in the odd and binary exact families it is
$\{1\}$.
\end{proof}

\subsection{A cofinite spectrum forces automorphy}

The following elementary criterion is included to delimit a repair that is
strong enough.

\begin{proposition}[Cofinite spectrum criterion]\label{prop:cofinite}
Let $F\in\F_q[X_1,\ldots,X_n]^n$.
\begin{enumerate}
\item If $\Specper_q(F)$ is cofinite in $\N_{>0}$, then
$\Specper_q(F)=\N_{>0}$.
\item If, in addition, $F$ is Keller, then the following are equivalent:
\begin{enumerate}
\item $F$ is a polynomial automorphism over $\F_q$;
\item $\Specper_q(F)=\N_{>0}$;
\item $\Specper_q(F)$ is cofinite.
\end{enumerate}
\end{enumerate}
\end{proposition}

\begin{proof}
Suppose every integer at least $M$ lies in the spectrum. For any $d\geq1$,
choose a multiple $kd\geq M$. By divisor-closure,
$kd\in\Specper_q(F)$ implies $d\in\Specper_q(F)$, which proves (1).

A polynomial automorphism is bijective after every scalar extension. For the
converse, assume that $F$ is Keller and bijective on every finite extension,
and base-change to $\overline{\F}_q$. Since
\[
 \overline{\F}_q=\bigcup_{m\geq1}\F_{q^m},
\]
the base-changed morphism is bijective on $\overline{\F}_q$-points. The
Keller condition makes it étale, hence quasi-finite and generically finite.
Over a nonempty open subset of the target it is finite étale of degree equal
to its geometric generic degree. Those geometric fibers are reduced, and
each contains exactly one point; hence the generic degree is one and the
morphism is birational.

By Zariski's Main Theorem, a birational quasi-finite separated morphism to
the normal variety $\A^n_{\overline{\F}_q}$ is an open immersion
\cite[Tag~02LQ]{StacksProject}. Its open
image contains every $\overline{\F}_q$-point. A nonempty closed complement
would contain a closed point, which over the algebraically closed ground
field is an $\overline{\F}_q$-point, a contradiction. Thus the base-changed
morphism is an isomorphism. Uniqueness of the inverse makes it invariant
under $\operatorname{Gal}(\overline{\F}_q/\F_q)$, so the inverse descends to
$\F_q$.
\end{proof}

\begin{remark}
The Keller assumption in Proposition~\ref{prop:cofinite}(2) is essential:
$T\mapsto T^p$ permutes every finite extension of $\F_p$ but is not a
polynomial automorphism of the affine line.
\end{remark}

\subsection{The prime-to-characteristic question}

For a dominant Keller map $F$, write
\[
 d_F=[\overline{\F}_q(X_1,\ldots,X_n):
       \overline{\F}_q(F_1,\ldots,F_n)]
\]
for its geometric generic degree. Every counterexample in this paper has
$p\mid d_F$.

\begin{question}[Prime-to-$p$ mock automorphisms]\label{question:primetop}
If $F$ is a mock automorphism over $\F_q$ and $p\nmid d_F$, must $F$ be a
polynomial automorphism?
\end{question}

In one variable the answer is yes. A polynomial with nonzero constant
derivative has the form
\[
 a+bT+H(T^p),
\]
so every nonlinear such polynomial has degree divisible by $p$; because the
derivative is nonzero, its geometric generic degree equals its polynomial
degree.

\subsection{Open problems}

The preceding results suggest the following concrete problems.

\begin{question}\label{question:even}
For even $q>2$, what is the exact spectrum of $v_q$ in \eqref{eq:vq}? Does
there exist a one-variable mock automorphism over every such $\F_q$ with
spectrum exactly $\{1\}$ and degree $p(q+1)$?
\end{question}

\begin{question}\label{question:delta}
Can $\delta_p$, and hence $\mu_p$, be expressed uniformly for primes in the
residual congruence classes such as $p\equiv1\pmod{105}$?
\end{question}

\begin{question}\label{question:globalmulti}
What is the least degree of a non-extension-exceptional mock automorphism of
$\F_q^n$ without the strong normalization $\Jac(F)=I_n$ and
$F|_{\F_q^n}=\id$?
\end{question}

\begin{question}\label{question:elementaryminimum}
Can the formula $\mu_p=p\delta_p$ for $p\geq5$ be proved without the degree
classification of indecomposable exceptional polynomials, for example by a
uniform collision-curve construction for every lift $f_P$ in
\eqref{eq:liftP}?
\end{question}

These questions record several concrete limitations of the present methods
and natural directions for extending the results.

\section{Conclusion}

Maubach--Willems Conjecture~2.7 fails in every positive characteristic and
over every finite field. In the explicit family $\Psi_{g,e,n}$, the spectra
of mock automorphisms are exactly $\N_{>0}$ and the finite nonempty
divisor-closed subsets of $\N_{>0}$. This gives a complete combinatorial
classification inside a sparse family of basic endomorphisms, rather than a
single isolated counterexample.

For every $n\geq2$, the least possible degree of a finite-spectrum map in
the strongly normalized class $\mathcal N_{q,n}$ is $p(q+1)$, and the bound
is attained by a spectrum-$\{1\}$ map. In dimension one, the same bound is
attained for odd $q$ and for $q=2$; the remaining even fields are open. The
prime-field minimum formula is classification-dependent outside the direct
families isolated in Section~\ref{sec:minimal}. Finiteness of the morphism
does not restore the conjecture, whereas a cofinite spectrum characterizes
polynomial automorphisms among Keller maps. The prime-to-$p$ generic-degree
condition is a natural surviving boundary for the constructions considered
here and remains unresolved in higher dimension.

\section*{Declaration of generative AI and AI-assisted technologies
in the manuscript preparation process}

During the preparation of this work, the authors used ChatGPT
(GPT-5.6 Sol) for exploratory mathematical brainstorming, manuscript
organization and language editing, LaTeX assistance, literature-search
support, and limited computational verification. The authors reviewed,
corrected, and independently verified all AI-assisted material, and take full responsibility for the content of
the published article.

\bibliographystyle{amsplain}
\bibliography{prescribed_extension_spectra_mock_automorphisms}

\end{document}